\documentclass[12pt,a4paper,reqno]{article}
\usepackage{graphics}
\usepackage{blkarray}
\usepackage{float}
\usepackage{color}
\usepackage{pdflscape}
\usepackage{lineno}
\usepackage{epsfig}

\usepackage{amsmath,amsthm,amssymb}
\allowdisplaybreaks
\newtheorem{theorem}{Theorem}[section]

\newtheorem{example}[theorem]{Example}
\theoremstyle{definition}

\theoremstyle{remark}

\numberwithin{equation}{section}

\begin{document}
\title{ Graph Theoretic and Spectral Properties of the Zero-Divisor Graph of
$\mathbb{F}_p + u\mathbb{F}_p + v\mathbb{F}_p + uv\mathbb{F}_p$}

\author{N. Annamalai\\
Assistant Professor\\
Department of Mathematics\\
University College of Engineering Arni\\
Thatchur-632 326,
 Tamil Nadu, India\\
{Email: algebra.annamalai@gmail.com}
}
\date{}
\maketitle
\begin{abstract}
In this article, we study the zero-divisor graph of the commutative ring with identity $R=
\mathbb{F}_p + u\mathbb{F}_p + v\mathbb{F}_p + uv\mathbb{F}_p,$
where $u^2 = 0, v^2 = 0, uv = vu$ and $p$ is an odd prime. 
We determine several graph-theoretic properties associated with the zero-divisor graph $\Gamma(R),$ including the clique number, chromatic number, vertex connectivity, edge connectivity, diameter and girth. In addition, we compute certain topological indices of the graph $\Gamma(R).$ Furthermore, we find the eigenvalues, energy and spectral radius of the adjacency matrix, the Laplacian matrix and the Eccentricity matrix of the zero-divisor graph $(\Gamma(R).$
\end{abstract}


\section{Introduction}

The study of zero-divisor graphs has received significant attention in the literature. The concept was first introduced by Beck \cite{beck} in 1988, where the additive identity $0$ of a ring $R$ was included in the vertex set. Beck's main focus was on investigating the coloring properties of commutative rings using the associated graph structure.

The concept of the zero-divisor graph was first introduced by Beck. Let $\Gamma$ be a simple graph whose vertices are the zero-divisors of a ring $R,$ where two distinct vertices are adjacent whenever their product is zero. Anderson and Livingston later refined this definition by considering only the nonzero zero-divisors of a commutative ring $R$ as vertices, while retaining the same adjacency criterion \cite{and}. Thus, the zero-divisor graph $\Gamma(R)$ is defined as the simple graph with vertex set $Z^{*}(R),$ where two distinct vertices $x$ and $y$ are adjacent if and only if $xy=0.$

Let $R$ be a commutative ring with identity, and let $Z(R)$ denote its set of zero-divisors. The zero-divisor graph of $R,$ denoted by $\Gamma(R),$ is the simple undirected graph with vertex set $Z^{*}(R),$ the set of all nonzero zero-divisors of $R.$ Two distinct vertices $x$ and $y$ are connected by an edge whenever $xy=0.$

In this paper, we investigate the zero-divisor graph $\Gamma(R)$ with vertex set $Z^{*}(R).$ Zero-divisor graphs have attracted considerable attention in recent years, leading to numerous important results and developments by various researchers \cite{am, anna, amir, kavaskar, red, reddy}.

Let $\Gamma=(V,E)$ be a simple undirected graph. The adjacency matrix $A(\Gamma)=[a_{ij}]$ of $\Gamma$ is the $|V|\times |V|$ matrix defined by
$$
a_{ij}=
\begin{cases}
1, & \text{if } i \text{ and } j \text{ are adjacent},\\
0, & \text{otherwise},
\end{cases}$$

with $a_{ii}=0$ for all $i\in V$. If $\lambda_1,\lambda_2,\ldots,\lambda_{|V|}$ are the eigenvalues of $A(\Gamma)$, then the energy of $\Gamma$ is given by

$$
\varepsilon(\Gamma)=\sum_{i=1}^{|V|} |\lambda_i|.
$$

Let $\Gamma=(V,E)$ be a simple undirected graph. The \emph{Laplacian matrix} of $\Gamma$, denoted by $L(\Gamma)$, is the $|V|\times |V|$ matrix whose rows and columns are indexed by the vertices of $\Gamma$. Its entries are defined by

$$
L(\Gamma)_{ij}=
\begin{cases}
-1, & \text{if } i \text{ and } j \text{ are adjacent},\\
0, & \text{if } i \text{ and } j \text{ are not adjacent},
\end{cases}
\qquad (i\neq j),
$$
and $L(\Gamma)_{ii}=d_i, \qquad i=1,2,\ldots,|V|,$
where $d_i$ denotes the degree of the vertex $i$.

Let $D(\Gamma)$ be the diagonal matrix of vertex degrees and let $A(\Gamma)$ be the adjacency matrix of $\Gamma$. Then $
L(\Gamma)=D(\Gamma)-A(\Gamma).$

Let $\mu_1,\mu_2,\ldots,\mu_{|V|}$ be the eigenvalues of $L(\Gamma)$. The \emph{Laplacian energy} of $\Gamma$ is defined by

$$
LE(\Gamma)=\sum_{i=1}^{|V|}\left|\mu_i-\frac{2|E|}{|V|}\right|.$$

Let $\Gamma(R)$ be a connected graph. The eccentricity matrix of $\Gamma(R)$, denoted by $E(\Gamma(R))$, is defined as $
E(\Gamma(R))=\bigl(e_{ij}\bigr),$
where $$
e_{ij}=
\begin{cases}
d(v_i,v_j), & \text{if } d(v_i,v_j)=\min\{e(v_i),e(v_j)\},\\
0, & \text{otherwise}.
\end{cases}
$$

Here, $d(v_i,v_j)$ denotes the distance between the vertices $v_i$ and $v_j$, and

$$e(v_i)=\max_{u\in V(\Gamma(R))} d(v_i,u)$$
is the eccentricity of the vertex $v_i$.

The \emph{Wiener index} of a connected graph $\Gamma$ is defined as the sum of the distances between all pairs of distinct vertices of $\Gamma$. That is,
$$
W(\Gamma)=\sum_{\substack{u,v\in V(\Gamma)\\ u<v}} d(u,v),$$

where $d(u,v)$ denotes the length of a shortest path between the vertices $u$ and $v$.

For a vertex $v\in V(\Gamma)$, the \emph{degree} of $v$, denoted by $d(v)$, is the number of edges incident with $v$. Equivalently, it is the number of vertices adjacent to $v$.

The \emph{Randi\'{c} index} of a graph $\Gamma$ is defined by 
$$R(\Gamma)=\sum_{ab\in E(\Gamma)} \frac{1}{\sqrt{d_a d_b}},$$

where $d_a$ and $d_b$ are the degrees of the end vertices of the edge $ab$. This index was introduced by Randi\'{c} \cite{randic} and is also known as the \emph{connectivity index}.

The \emph{Zagreb indices} are among the oldest and most extensively studied degree-based topological indices. They were introduced by Gutman and Trinajesti\'{c} \cite{gutman}. For a graph $\Gamma$, the first Zagreb index and the second Zagreb index, denoted by $M_1(\Gamma)$ and $M_2(\Gamma)$, respectively, are defined as
$$
M_1(\Gamma)=\sum_{v\in V(\Gamma)} d(v)^2,$$
and
$$
M_2(\Gamma)=\sum_{uv\in E(\Gamma)} d(u)d(v),$$
where $d(v)$ denotes the degree of the vertex $v$.

Let $\Gamma=(V,E)$ be a connected graph. An \emph{edge-cut} of $\Gamma$ is a subset $S\subseteq E$ such that the graph
$\Gamma-S=(V,E\setminus S)$
is disconnected. The \emph{edge-connectivity} of $\Gamma$, denoted by $\lambda(\Gamma)$, is the minimum cardinality of an edge-cut of $\Gamma$.

Similarly, a \emph{vertex-cut} of $\Gamma$ is a subset $T\subseteq V$ whose removal disconnects the graph. The minimum cardinality of a vertex-cut is called the \emph{vertex connectivity}, or simply the \emph{connectivity}, of $\Gamma$, and is denoted by $\kappa(\Gamma).$

For any connected graph $\Gamma$, the edge-connectivity of $\Gamma$ satisfies
$\lambda(\Gamma)\leq \delta(\Gamma),$
 where $\delta(\Gamma)$ denotes the minimum degree among all vertices of $\Gamma$.

The \emph{chromatic number} of a graph $\Gamma$, denoted by $\chi(\Gamma)$, is the minimum number of colors required to color the vertices of $\Gamma$ such that no two adjacent vertices receive the same color.

The \emph{clique number} of a graph $\Gamma$, denoted by $\omega(\Gamma)$, is the maximum cardinality of a clique in $\Gamma$, that is,
$$
\omega(\Gamma)=\max\{|C| : C\subseteq V(\Gamma)\ \text{and}\ uv\in E(\Gamma)\ \text{for all distinct}\ u,v\in C\}.$$

Equivalently, $\omega(\Gamma)$ is the order of a largest complete subgraph of $\Gamma$. In the particular case of a zero-divisor graph, a subset $C\subseteq V(\Gamma)$ forms a clique if and only if $xy=0$ for all distinct $x,y\in C$.
It is well known that $\omega(\Gamma)\leq \chi(\Gamma)$
for every graph $\Gamma$. 

 Beck \cite{beck} conjectured that for any finite ring $R$ with finite chromatic number,
$
\omega(\Gamma(R))=\chi(\Gamma(R)),$
where $\omega(\Gamma(R))$ and $\chi(\Gamma(R))$ denote the clique number and chromatic number of the zero-divisor graph $\Gamma(R)$, respectively. He verified this conjecture for several classes of rings.

However, Anderson and Naseer \cite{and} disproved Beck's conjecture by providing a counterexample. Subsequently, the author showed that the clique number and chromatic number of the zero-divisor graph associated with the ring
$\mathbb{F}_p+u\mathbb{F}_p+u^2\mathbb{F}_p$
coincide, that is,
$
\omega(\Gamma(R))=\chi(\Gamma(R))$
for this class of rings \cite{am}.
For standard terminology and results in graph theory, the reader is referred to \cite{R.B,bapat}.

Let $p$ be an odd prime and let $
R=\mathbb{F}_p+u\mathbb{F}_p+v\mathbb{F}_p+uv\mathbb{F}_p,$
where $u^2=v^2=0$ and $uv=vu$. Then $R$ is a finite commutative ring of characteristic $p$ and
$
R\cong \frac{\mathbb{F}_p[u,v]}{\langle u^2,v^2,uv-vu\rangle}.$

Each element of $R$ has a unique representation
$
a+ub+vc+uvd,\qquad a,b,c,d\in\mathbb{F}_p.$
Moreover, an element $a+ub+vc+uvd$ is invertible in $R$ if and only if $a\neq 0$.

Throughout this article, we denote the ring
$\mathbb{F}_p+u\mathbb{F}_p+v\mathbb{F}_p+uv\mathbb{F}_p$
by $R$, where $u^2=0$, $v^2=0$, and $uv=vu$. We investigate the zero-divisor graph $\Gamma(R)$ associated with this commutative ring with identity. In Section~2, we determine several graph-theoretic properties of $\Gamma(R)$, including its clique number, chromatic number, vertex connectivity, edge connectivity, diameter, and girth. In Section~3, we compute various topological indices of $\Gamma(R)$. Finally, in Section~4, we study the spectral properties of $\Gamma(R)$ by determining the eigenvalues, spectral radius, and energy of its adjacency, Laplacian, and eccentricity matrices.

\section{Zero-divisor graph $\Gamma(R)$ of the ring $R$}
In this section, we study the zero-divisor graph $\Gamma(R)$ of the ring $R$ and determine some of its fundamental graph-theoretic properties. Specifically, we compute the clique number, chromatic number, vertex connectivity, edge connectivity, diameter, and girth of $\Gamma(R)$.

Let
$A_u=\{x u\mid x\in \mathbb{F}_p^{*}\},$
$A_{v}=\{x v\mid x\in \mathbb{F}_p^{*}\},$ $A_{uv}=\{x uv\mid x\in \mathbb{F}_p^{*}\},$ $A_{u+v}=\{xu+yv\mid x, y\in \mathbb{F}_p^{*}\},$ 
 $A_{u+uv}=\{xu+yuv\mid x, y\in \mathbb{F}_p^{*}\},$ $A_{v+uv}=\{xv+yuv\mid x, y\in \mathbb{F}_p^{*}\}$ and $A_{u+v+uv}=\{xu+yv+z uv\mid x, y, z\in \mathbb{F}_p^{*}\}.$    Then $|A_u|=(p-1),$ $|A_{v}|=(p-1),$ $|A_{uv}|=(p-1),$ $|A_{u+v}|=(p-1)^2,$ $|A_{u+uv}|=(p-1)^2,$ $|A_{v+uv}|=(p-1)^2$ and $|A_{u+v+uv}|=(p-1)^3.$
Therefore, 
\begin{align*}
Z^{*}(R)&=A_u\cup A_{v}\cup A_{uv}\cup A_{u+v}\cup A_{u+uv}\cup A_{v+uv}\cup A_{u+v+uv}\\
|Z^{*}(R)|&=|A_u|+|A_{v}|+|A_{uv}|+|A_{u+v}|+|A_{u+uv}|+|A_{v+uv}|+|A_{u+v+uv}|\\&=(p-1)+(p-1)+(p-1)+(p-1)^2+(p-1)^2+(p-1)^2+(p-1)^3\\
&=(p-1)(p^2+p+1)\\
&=p^3-1.
\end{align*}
\begin{figure}[H]
  \begin{center}
\includegraphics{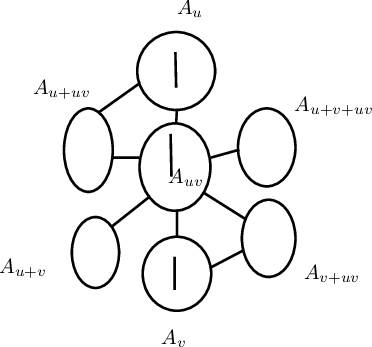}
  \end{center}
  \caption{Zero-divisor graph of $R=\mathbb{F}_p+u\mathbb{F}_p+v\mathbb{F}_p+uv\mathbb{F}_p$}
\end{figure}
Since $u^2=0$, $v^2=0$, and $uv=vu$, every vertex of $A_{uv}$ is adjacent to every vertex of
$
A_u,\; A_v,\; A_{u+v},\; A_{u+uv},\; A_{v+uv}, \text{ and } A_{u+v+uv}.$

Moreover, every vertex of $A_u$ is adjacent to every vertex of $A_{u+uv}$, and every vertex of $A_v$ is adjacent to every vertex of $A_{v+uv}$.
Furthermore, any two distinct vertices of $A_u$ are adjacent, any two distinct vertices of $A_v$ are adjacent, and any two distinct vertices of $A_{uv}$ are adjacent. Hence, the induced subgraphs on $A_u$, $A_v$, and $A_{uv}$ are complete graphs.

It follows from the above adjacency relations that the zero-divisor graph $\Gamma(R)$ is connected. The graph $\Gamma(R)$ has
$
p^{3}-1$
vertices and
$
(p-1)^4+5(p-1)^3+2(p-1)^2+\frac{3(p-1)(p-2)}{2}
$
edges. Simplifying the above expression, we obtain
$$
|E(\Gamma(R))|
=\frac{1}{2}\left(2p^{4}+2p^{3}-11p^{2}+5p+2\right).$$

\begin{example}\label{a}
For $p=3,$ $R=\mathbb{F}_3+u\mathbb{F}_3+v\mathbb{F}_3+uv\mathbb{F}_3.$ Then
$A_u=\{u, 2u\},$ $A_{v}=\{v, 2v\},$ $A_{uv}=\{uv, 2uv\}$ $A_{u+v}=\{ u+v, 2u+2v, u+2v, 2u+v\},$ $A_{u+uv}=\{ u+uv, 2u+2uv, u+2uv, 2u+uv\},$ $A_{v+uv}=\{ v+uv, 2v+2uv, v+2uv, 2v+uv\},$ $A_{u+v+uv}=\{ u+v+uv, 2u+2v+uv, u+2v+uv, 2u+v+uv, u+v+2uv, 2u+v+2uv, u+2v+2uv, 2u+2v+2uv\}.$
\begin{figure}[H]
  \begin{center}
\includegraphics{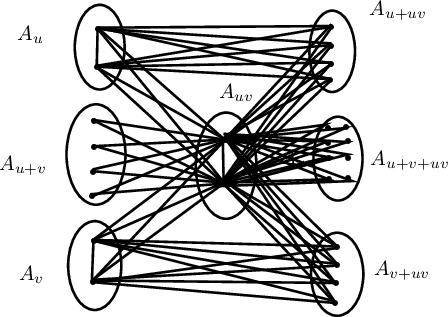}\\
  \end{center}
  \caption{Zero-divisor graph of $R=\mathbb{F}_3+u\mathbb{F}_3+v\mathbb{F}_3+uv\mathbb{F}_3$}
\end{figure}
The zero-divisor graph $\Gamma(R)$ has
$
|V(\Gamma(R))|=3^3-1=26$
vertices and $
|E(\Gamma(R))|
=\frac{1}{2}\left(2(3)^4+2(3)^3-11(3)^2+5(3)+2\right)
=67$
edges. Thus, $\Gamma(R)$ is a connected graph with $26$ vertices and $67$ edges.

\end{example}

\begin{theorem}
The diameter of the zero-divisor graph $diam(\Gamma(R))=2.$
\end{theorem}
\begin{proof}
From Figure~1, it is evident that the distance between any two distinct vertices of $\Gamma(R)$ is either $1$ or $2$. Consequently, the maximum distance between any pair of distinct vertices is $2$. Therefore,
$
\operatorname{diam}(\Gamma(R))=2.$
\end{proof}
\begin{theorem}
The clique number $\omega(\Gamma(R))$ of $\Gamma(R)$ is $2p-2.$
\end{theorem}
\begin{proof}
From Figure~1, it follows that both $A_{uv}\cup A_u$ and $A_{uv}\cup A_v$ induce complete subgraphs (cliques) of $\Gamma(R)$. Moreover, neither of these cliques can be enlarged by including any additional vertex of $\Gamma(R)$. Hence, they are maximum cliques in $\Gamma(R)$.

Therefore, the clique number of $\Gamma(R)$ is
$$
\omega(\Gamma(R))
=|A_{uv}\cup A_u|
=|A_{uv}\cup A_v|
=|A_{uv}|+|A_u|
=(p-1)+(p-1)
=2p-2.$$
\end{proof}

\begin{theorem}
The chromatic number $\chi(\Gamma(R))$ of $\Gamma(R)$ is $2p-1.$
\end{theorem}
\begin{proof}
Since $A_{uv}\cup A_u$ (and similarly $A_{uv}\cup A_v$) induces a complete subgraph of $\Gamma(R)$ with $2p-2$ vertices, at least $2p-2$ distinct colors are required in any proper coloring of $\Gamma(R)$.

Observe that there are no edges between vertices of $A_u$ and $A_v$. Hence, the colors assigned to the vertices of $A_u$ can be reused to color the vertices of $A_v$.

Furthermore, each of the sets $
A_{u+v},\quad A_{u+uv},\quad A_{v+uv},\quad \text{and} \quad A_{u+v+uv}$ 
is an independent set. Therefore, all vertices belonging to these sets may be assigned a single additional color.

Consequently, a proper coloring of $\Gamma(R)$ can be obtained using $
(2p-2)+1=2p-1$
colors. Since at least $2p-2$ colors are necessary because of the maximum clique and one additional color is required for the remaining vertices, it follows that $
\chi(\Gamma(R))=2p-1.$
\end{proof}

As a consequence of the above two theorems, we have

$$
\omega(\Gamma(R))=2p-2<2p-1=\chi(\Gamma(R)).$$

Thus, the clique number and the chromatic number of the zero-divisor graph $\Gamma(R)$ do not coincide.

\begin{theorem}
The girth of the graph $\Gamma(R)$ is 3.
\end{theorem}
\begin{proof}
From Figure~1, the sets $A_u$, $A_v$, and $A_{uv}$ each induce a complete subgraph of $\Gamma(R)$. Since every complete graph with at least three vertices contains a triangle, $\Gamma(R)$ contains a cycle of length $3$. Therefore,
$g(\Gamma(R))=3.$
\end{proof}

\begin{theorem}
The vertex connectivity $\kappa(\Gamma(R))$ of $\Gamma(R)$ is $p-1.$
\end{theorem}
\begin{proof}
Since every vertex of $\Gamma(R)$ has degree at least $p-1$, we have
$\delta(\Gamma(R))=p-1,$
where $\delta(\Gamma(R))$ denotes the minimum degree of $\Gamma(R)$. Moreover, the removal of all vertices of $A_{uv}$ disconnects the graph, and
$
|A_{uv}|=p-1.$

Thus, there exists a vertex cut of cardinality $p-1$. Consequently,
$
\kappa(\Gamma(R))\leq p-1.$
On the other hand, at least $p-1$ vertices must be removed to disconnect $\Gamma(R)$. Therefore,
$
\kappa(\Gamma(R))\geq p-1.$
Hence, $\kappa(\Gamma(R))=p-1.$
\end{proof}

\begin{theorem}
The edge connectivity $\lambda(\Gamma(R))$ of $\Gamma(R)$ is $p-1.$
\end{theorem}
\begin{proof}
Since $\Gamma(R)$ is connected, we have
$
\kappa(\Gamma(R))
\leq
\lambda(\Gamma(R))
\leq
\delta(\Gamma(R)).$
Since $\kappa(\Gamma(R))=\delta(\Gamma(R))=p-1$, it follows that
$
p-1
\leq
\lambda(\Gamma(R))
\leq
p-1.$
Hence, $
\lambda(\Gamma(R))=p-1.$
\end{proof}

\section{Some Topological Indices of $\Gamma(R)$}
In this section, we study some degree-based and distance-based topological indices of the zero-divisor graph $\Gamma(R)$. Specifically, we derive explicit formulas for the Wiener index, the first Zagreb index, the second Zagreb index, and the Randi\'{c} index of $\Gamma(R)$.

\begin{theorem}
The Wiener index of the zero-divisor graph $\Gamma(R)$ of $R$ is
$$\dfrac{2p^6-26p^4+32p^3-7p^2-13p+12}{2}.$$
\end{theorem}
\begin{proof}
Since every vertex of $A_{uv}$ connects to every vertex of the graph, the distance between any two vertices is 1 or 2.
Consider
\begin{align*}
    W(\Gamma(R))&=\sum_{\substack{x, y\in Z^{*}(R)\\ x\neq y}} d(x, y)\\
    &=\sum_{\substack{x, y\in Z^{*}(R)\\ x\neq y\\d(x,y)=1}} d(x, y)+\sum_{\substack{x, y\in Z^{*}(R)\\ x\neq y\\d(x,y)=2}} d(x, y)\\
    &=\Big\{3\dfrac{(p-1)(p-2)}{2}+(p-1)^4+5(p-1)^3+2(p-1)^2\Big\}(1)\\
    &+\Big\{ \binom{(p-1)^3}{2}+3 \binom{(p-1)^2}{2}+3(p-1)^5+5(p-1)^4+4(p-1)^3+(p-1)^2\Big\}(2)\\
    &=\dfrac{2p^4+2p^3-11p^2+5p+2}{2}+p^6-14p^4+15p^3+2p^2-9p+5\\
    &=\dfrac{2p^6-26p^4+32p^3-7p^2-13p+12}{2}.
\end{align*}
\end{proof}
Let $[A,B]$ denote the set of edges with one end vertex in $A$ and the other in $B$, where $A,B\subseteq V(\Gamma(R))$.

The degrees of the vertices in $\Gamma(R)$ are given as follows:
$$
d(a)=
\begin{cases}
p^2-2, & \text{if } a\in A_u,\\[1mm]
p^2-2, & \text{if } a\in A_v,\\[1mm]
p-1, & \text{if } a\in A_{u+v},\\[1mm]
2p-2, & \text{if } a\in A_{u+uv},\\[1mm]
2p-2, & \text{if } a\in A_{v+uv},\\[1mm]
p-1, & \text{if } a\in A_{u+v+uv},\\[1mm]
p^3-2, & \text{if } a\in A_{uv}.
\end{cases}
$$

\begin{theorem}
The Randi\'{c} index of the zero-divisor graph $\Gamma(R)$ of $R$ is
$R(\Gamma(R))=(p-1)(p-2)\left(\frac{1}{p^2-2}+\frac{1}{2(p^3-2)}\right)
+\frac{2(p-1)^2}{\sqrt{(p^2-2)(p^3-2)}}+\frac{(p-1)^3}{\sqrt{2(p-1)}}
\left(
\frac{3}{\sqrt{p^3-2}}
+\frac{2}{\sqrt{p^2-2}}
\right)+\frac{(p-1)^4}{\sqrt{(p-1)(p^3-2)}}.$
\end{theorem}
\begin{proof}
Clearly, there is no direct edge between the vertices of $A_{u+v}, \,A_{u+uv},\, A_{v+uv}, A_{u+v+uv}.$
Then,
\begin{align*}
    R(\Gamma(R))&=\sum\limits_{(a,b)\in E} \frac{1}{\sqrt{d_a d_b}}\\
    &=\sum\limits_{(a,b)\in [A_u, A_{u}]} \frac{1}{\sqrt{d_a d_b}}+\sum\limits_{(a,b)\in [A_{u}, A_{uv}]} \frac{1}{\sqrt{d_a d_b}}+\sum\limits_{(a,b)\in [A_{v}, A_{v}]} \frac{1}{\sqrt{d_a d_b}}\\
    &\hspace{0.5cm}+\sum\limits_{(a,b)\in [A_v, A_{uv}]} \frac{1}{\sqrt{d_a d_b}}+\sum\limits_{(a,b)\in [A_{uv}, A_{uv}]} \frac{1}{\sqrt{d_a d_b}}+\sum\limits_{(a,b)\in [A_{u+v}, A_{uv}]} \frac{1}{\sqrt{d_a d_b}}\\
    &\hspace{0.5cm}+\sum\limits_{(a,b)\in [A_{u+uv}, A_{uv}]} \frac{1}{\sqrt{d_a d_b}}+\sum\limits_{(a,b)\in [A_{v+uv}, A_{uv}]} \frac{1}{\sqrt{d_a d_b}}\\&\hspace{0.5cm}+\sum\limits_{(a,b)\in [A_{u+v+uv}, A_{uv}]} \frac{1}{\sqrt{d_a d_b}}
    +\sum\limits_{(a,b)\in [A_{u+uv}, A_{u}]} \frac{1}{\sqrt{d_a d_b}}\\&\hspace{0.5cm}+\sum\limits_{(a,b)\in [A_{v+uv}, A_{v}]} \frac{1}{\sqrt{d_a d_b}}\\
    &=\binom{p-1}{2} \frac{1}{\sqrt{(p^2-2)^2}}
    +(p-1)^2\frac{1}{\sqrt{(p^2-2)(p^3-2)}}
    +\binom{p-1}{2} \frac{1}{\sqrt{(p^2-2)^2}}\\
&\hspace{0.5cm}+(p-1)^2\frac{1}{\sqrt{(p^2-2)(p^3-2)}}+\binom{p-1}{2}\frac{1}{\sqrt{(p^3-2)^2}}\\
&\hspace{0.5cm}+(p-1)^3\frac{1}{\sqrt{(2p-2)(p^3-2)}}
+(p-1)^3\frac{1}{\sqrt{(2p-2)(p^3-2)}}\\
&\hspace{0.5cm}+(p-1)^3\frac{1}{\sqrt{(2p-2)(p^3-2)}}+(p-1)^4\frac{1}{\sqrt{(p-1)(p^3-2)}}\\
&\hspace{0.5cm}+(p-1)^3\frac{1}{\sqrt{(2p-2)(p^2-2)}}+(p-1)^3\frac{1}{\sqrt{(2p-2)(p^2-2)}}\\
&=(p-1)(p-2)\left(\frac{1}{p^2-2}+\frac{1}{2(p^3-2)}\right)
+\frac{2(p-1)^2}{\sqrt{(p^2-2)(p^3-2)}}\\
&\hspace{0.5cm}+\frac{(p-1)^3}{\sqrt{2(p-1)}}
\left(
\frac{3}{\sqrt{p^3-2}}
+\frac{2}{\sqrt{p^2-2}}
\right)+\frac{(p-1)^4}{\sqrt{(p-1)(p^3-2)}}
    \end{align*}
\end{proof}

\begin{theorem}
The first Zagreb index of the zero-divisor graph $\Gamma(R)$ of $R$ is
$$M_1(\Gamma(R))=p^7-p^6+2p^5+8p^4-12p^3-4p^2+26p-14.$$
\end{theorem}
\begin{proof}
Consider
\begin{align*}
    M_1(\Gamma(R))&=\sum\limits_{a\in Z^{*}(R)} d_a^2\\
    &=\sum_{a\in A_u}d_a^2+\sum_{a\in A_v}d_a^2+\sum_{a\in A_{uv}}d_a^2+\sum_{a\in A_{u+v}}d_a^2+\sum_{a\in A_{u+uv}}d_a^2+\sum_{a\in A_{v+uv}}d_a^2+\sum_{a\in A_{u+v+uv}}d_a^2\\
    &=(p-1)(p^2-2)^2+(p-1)(p^2-2)^2+(p-1)(p^3-2)^2+(p-1)^2(p-1)^2\\
    &+(p-1)^2(2p-2)^2+(p-1)^2(2p-2)^2+(p-1)^3(p-1)\\
    &=p^7-p^6+2p^5+8p^4-12p^3-4p^2+26p-14.
    \end{align*}
\end{proof}
\begin{theorem}
The second Zagreb index of the zero-divisor graph $\Gamma(R)$ of $R$ is
$$M_2(\Gamma(R))=(p-1)\left(
p^7+6p^6-\frac{43}{2}p^5+\frac{47}{2}p^4
+3p^3-5p^2+24p-4
\right).$$
\end{theorem}
\begin{proof}
Consider
\begin{align*}
    M_2(\Gamma(R))&=\sum\limits_{(a,b)\in E} d_a d_b\\
    &=\sum\limits_{(a,b)\in [A_u, A_{u}]} d_a d_b+\sum\limits_{(a,b)\in [A_{u}, A_{uv}]} d_a d_b+\sum\limits_{(a,b)\in [A_{v}, A_{v}]} d_a d_b\\
    &\hspace{0.5cm}+\sum\limits_{(a,b)\in [A_v, A_{uv}]} d_a d_b+\sum\limits_{(a,b)\in [A_{uv}, A_{uv}]} d_a d_b+\sum\limits_{(a,b)\in [A_{u+v}, A_{uv}]} d_a d_b\\
    &\hspace{0.5cm}+\sum\limits_{(a,b)\in [A_{u+uv}, A_{uv}]} d_a d_b+\sum\limits_{(a,b)\in [A_{v+uv}, A_{uv}]} d_a d_b+\sum\limits_{(a,b)\in [A_{u+v+uv}, A_{uv}]} d_a d_b\\
    &\hspace{0.5cm}+\sum\limits_{(a,b)\in [A_{u+uv}, A_{u}]} d_a d_b+\sum\limits_{(a,b)\in [A_{v+uv}, A_{v}]} d_a d_b\\
    &=\binom{p-1}{2} (p^2-2)^2
    +(p-1)^2(p^2-2)(p^3-2)
    +\binom{p-1}{2} (p^2-2)^2\\
&\hspace{0.5cm}+(p-1)^2(p^2-2)(p^3-2)+\binom{p-1}{2}(p^3-2)^2+(p-1)^3(2p-2)(p^3-2)\\
&\hspace{0.5cm}+(p-1)^3(2p-2)(p^3-2)
+(p-1)^3(2p-2)(p^3-2)\\
&\hspace{0.5cm}+(p-1)^4(p-1)(p^3-2)+(p-1)^3(2p-2)(p^2-2)\\&\hspace{0.5cm}+(p-1)^3(2p-2)(p^2-2)\\
&=(p-1)\left(
p^7+6p^6-\frac{43}{2}p^5+\frac{47}{2}p^4
+3p^3-5p^2+24p-4
\right).
    \end{align*}
\end{proof}
\section{Adjacency, Laplacian and Eccentricity  Matrices of $\Gamma(R)$}
In this section, we study the spectra of the adjacency, Laplacian, and eccentricity matrices associated with the zero-divisor graph $\Gamma(R)$. We derive their eigenvalues and compute the corresponding graph energies and spectral radii.

Let $A$ be a matrix. The spectrum of $A$ is denoted by

\[
\operatorname{Spec}(A)
=
\left(
\begin{array}{cccc}
\mu_1 & \mu_2 & \cdots & \mu_n\\
k_1 & k_2 & \cdots & k_n
\end{array}
\right),
\]

where each $\mu_i$ is a distinct eigenvalue of $A$ and $k_i$ is its algebraic multiplicity.

The vertex set partitions into $A_u,\, A_{v}, \, A_{uv},\, A_{u+v}, \, A_{u+uv},\, A_{v+uv}$ and $A_{u+v+uv}$ of cardinality $p-1,\, p-1, \, p-1, \, (p-1)^2,\, (p-1)^2,\, (p-1)^2$ and $(p-1)^3,$ respectively.
Then the adjacency matrix of $\Gamma(R)$ is

$A(\Gamma(R)) =\bordermatrix{
      & A_u & A_v & A_{uv} & A_{u+v} &A_{u+uv}&A_{v+uv}&A_{u+v+uv}\cr
A_u & J-I & {\bf 0} & J & {\bf0}& J&{\bf0}&{\bf0} \cr
A_v & {\bf0} & J-I & J & {\bf0} & {\bf0}&J&{\bf0}\cr
A_{uv} & J & J & J-I & J& J&J&J\cr
A_{u+v} & {\bf 0} & {\bf0}& J & {\bf 0}& {\bf0}&{\bf0}&{\bf0}\cr
A_{u+uv} & J& {\bf0} & J & {\bf 0}& {\bf0}&{\bf0}&{\bf0}\cr
A_{v+uv} & {\bf 0} & J& J& {\bf 0}& {\bf0}&{\bf0}&{\bf0}\cr
A_{u+v+uv} & {\bf 0} & {\bf0} & J& {\bf 0}& {\bf0}&{\bf0}&{\bf0}\cr
}$

where $J$ is a all one matrix,  ${\bf 0}$ is a zero matrix in respective order and $I$ is an identity matrix.

Since each block of the adjacency matrix is either $J$, $J-I$, or $0$, the rows corresponding to vertices within the same partition class are linearly dependent, except for the contribution from the diagonal blocks of the form $J-I$. Consequently, the independent directions arise from:
\begin{enumerate}
\item one all-one vector associated with each of the seven partition classes; and

\item two additional independent directions contributed by the three diagonal blocks of the form $J-I$.
\end{enumerate}

It follows that $
\operatorname{rank}(A(\Gamma(R)))=3p.$
Since $\Gamma(R)$ has $p^3-1$ vertices, the adjacency matrix $A(\Gamma(R))$ is of order $p^3-1$. Therefore, by the Rank--Nullity Theorem,
$$
\operatorname{nullity}(A(\Gamma(R)))
=(p^3-1)-3p
=p^3-3p-1.$$
Hence, $0$ is an eigenvalue of $A(\Gamma(R))$ with multiplicity
$
p^3-3p-1.$
\begin{theorem}
The eigen values of the adjacency matrix $A(\Gamma(R))$ are 
\[
\operatorname{Spec}(A(\Gamma(R)))
=
\left(
\begin{array}{cccccccc}
-1 &0 & \lambda_1&\lambda_2&\lambda_3&\cdots& \lambda_7\\
3(p-2) & p^2-3p-2 &1&1&1& \cdots & 1
\end{array}
\right)
\] 

where $\lambda_1,\lambda_2,\lambda_3,\cdots, \lambda_7$ are eigen values of the $7\times 7$ quotient matrix
\[
Q=
\begin{bmatrix}
p-2 & 0 & p-1 & 0 & (p-1)^2 & 0 & 0\\
0 & p-2 & p-1 & 0 & 0 & (p-1)^2 & 0\\
p-1 & p-1 & p-2 & (p-1)^2 & (p-1)^2 & (p-1)^2 & (p-1)^3\\
0 & 0 & p-1 & 0 & 0 & 0 & 0\\
p-1 & 0 & p-1 & 0 & 0 & 0 & 0\\
0 & p-1 & p-1 & 0 & 0 & 0 & 0\\
0 & 0 & p-1 & 0 & 0 & 0 & 0
\end{bmatrix}.
\]
\end{theorem}
\begin{example}
For $p=3,$ the adjacency matrix of $\Gamma(R)$ is
$$A(\Gamma(R))= \bordermatrix{ & A_u^{(2)} & A_v^{(2)} & A_{uv}^{(2)} & A_{u+v}^{(4)} & A_{u+uv}^{(4)} & A_{v+uv}^{(4)} & A_{u+v+uv}^{(8)} \cr A_u^{(2)} & J_2-I_2 & {\bf 0}_{2\times2} & J_{2\times2} & {\bf 0}_{2\times4} & J_{2\times4} & {\bf 0}_{2\times4} & {\bf 0}_{2\times8} \cr A_v^{(2)} & {\bf 0}_{2\times2} & J_2-I_2 & J_{2\times2} & {\bf 0}_{2\times4} & {\bf 0}_{2\times4} & J_{2\times4} & {\bf 0}_{2\times8} \cr A_{uv}^{(2)} & J_{2\times2} & J_{2\times2} & J_2-I_2 & J_{2\times4} & J_{2\times4} & J_{2\times4} & J_{2\times8} \cr A_{u+v}^{(4)} & {\bf 0}_{4\times2} & {\bf 0}_{4\times2} & J_{4\times2} & {\bf 0}_{4\times4} & {\bf 0}_{4\times4} & {\bf 0}_{4\times4} & {\bf 0}_{4\times8} \cr A_{u+uv}^{(4)} & J_{4\times2} & {\bf 0}_{4\times2} & J_{4\times2} & {\bf 0}_{4\times4} & {\bf 0}_{4\times4} & {\bf 0}_{4\times4} & {\bf 0}_{4\times8} \cr A_{v+uv}^{(4)} & {\bf 0}_{4\times2} & J_{4\times2} & J_{4\times2} & {\bf 0}_{4\times4} & {\bf 0}_{4\times4} & {\bf 0}_{4\times4} & {\bf 0}_{4\times8} \cr A_{u+v+uv}^{(8)} & {\bf 0}_{8\times2} & {\bf 0}_{8\times2} & J_{8\times2} & {\bf 0}_{8\times4} & {\bf 0}_{8\times4} & {\bf 0}_{8\times4} & {\bf 0}_{8\times8} \cr }_{26\times26}$$
The eigenvalues of $A(\Gamma(R))$ are

\[
\operatorname{Spec}(A(\Gamma(R)))
=
\left(
\begin{array}{cccccccc}
-1 &0 & \lambda_1&\lambda_2&\lambda_3&\cdots& \lambda_7\\
3 & 16 &1&1&1& \cdots & 1
\end{array}
\right)
\]
 where $\lambda_1,\lambda_2,\lambda_3,\cdots, \lambda_7$ are eigen values of the quotient matrix
\[
Q=
\begin{bmatrix}
1 & 0 & 2 & 0 & 4 & 0 & 0\\
0 & 1 & 2 & 0 & 0 & 4 & 0\\
2 & 2 & 1 & 4 & 4 & 4 & 8\\
0 & 0 & 2 & 0 & 0 & 0 & 0\\
2 & 0 & 2 & 0 & 0 & 0 & 0\\
0 & 2 & 2 & 0 & 0 & 0 & 0\\
0 & 0 & 2 & 0 & 0 & 0 & 0
\end{bmatrix}.
\]
 \end{example}
The degree matrix of the graph $\Gamma(R)$ is
$$D(\Gamma(R)) =\begin{bmatrix}
 (p^2-2)I & {\bf 0} & {\bf 0} & {\bf0}& {\bf 0}&{\bf0}&{\bf0} \\
 {\bf0} & (p^2-2)I & {\bf 0} & {\bf0} & {\bf0}&{\bf 0}&{\bf0}\\
{\bf 0} & {\bf 0} & (p^3-2)I & {\bf 0}& {\bf 0}&{\bf 0}&{\bf 0}\\
 {\bf 0} & {\bf0}& {\bf 0} & (p-1)I& {\bf0}&{\bf0}&{\bf0}\\
 {\bf 0}& {\bf0} & {\bf 0} & {\bf 0}& (2p-2)I&{\bf 0}&{\bf0}\\
 {\bf 0} & {\bf 0}& {\bf 0}& {\bf 0}& {\bf0}&(2p-2)I&{\bf0}\\
 {\bf 0} & {\bf0} & {\bf 0}& {\bf 0}& {\bf0}&{\bf0}&(p-1)I
\end{bmatrix}
$$
The Laplacian matrix $L(\Gamma(R))$ of $\Gamma(R)$ is defined by $L(\Gamma(R))=D(\Gamma(R))-A(\Gamma(R)).$ Therefore,

{\scriptsize\[
L(\Gamma(R))=
\bordermatrix{
 & A_u & A_v & A_{uv} & A_{u+v} & A_{u+uv} & A_{v+uv} & A_{u+v+uv}\cr
A_u & (p^2-1)I-J & {\bf 0} & -J & {\bf0} & -J & {\bf0} & {\bf0} \cr
A_v & {\bf0} & (p^2-1)I-J & -J & {\bf0} & {\bf0} & -J & {\bf0}\cr
A_{uv} & -J & -J & (p^3-1)I-J & -J & -J & -J & -J\cr
A_{u+v} & {\bf0} & {\bf0} & -J & (p-1)I & {\bf0} & {\bf0} & {\bf0}\cr
A_{u+uv} & -J & {\bf0} & -J & {\bf0} & (2p-2)I & {\bf0} & {\bf0}\cr
A_{v+uv} & {\bf0} & -J & -J & {\bf0} & {\bf0} & (2p-2)I & {\bf0}\cr
A_{u+v+uv} & {\bf0} & {\bf0} & -J & {\bf0} & {\bf0} & {\bf0} & (p-1)I
}.
\]}

where $I$ denotes identity matrices of suitable orders,
$J$ denotes all-one matrices of suitable orders,
$0$ denotes zero matrices.

For $p=3,$ the Laplacian matrix is
$$L(\Gamma(R))=\bordermatrix{ & A_u & A_v & A_{uv} & A_{u+v} &A_{u+uv}&A_{v+uv}&A_{u+v+uv}\cr A_u & 8I-J & {\bf 0} & -J & {\bf0}& -J&{\bf0}&{\bf0} \cr A_v & {\bf0} & 8I-J & -J & {\bf0} & {\bf0}&-J&{\bf0}\cr A_{uv} & -J & -J & 26I-J & -J& -J&-J&-J\cr A_{u+v} & {\bf 0} & {\bf0}& -J & 2I& {\bf0}&{\bf0}&{\bf0}\cr A_{u+uv} & -J& {\bf0} & -J & {\bf 0}& 4I&{\bf0}&{\bf0}\cr A_{v+uv} & {\bf 0} & -J& -J& {\bf 0}& {\bf0}&4I&{\bf0}\cr A_{u+v+uv} & {\bf 0} & {\bf0} & -J& {\bf 0}& {\bf0}&{\bf0}&2I }.$$
The eigenvalues of $L(\Gamma(R))$ 
\[
\operatorname{Spec}_L(\Gamma(R))
=
\left(
\begin{array}{ccccc}
0 & 2 & 4 & 8 & 26\\
1 & 13 & 6 & 4 & 2
\end{array}
\right)
\]

For any prime $p,$ the eigenvalues of $L(\Gamma(R))$
\[
\operatorname{Spec}_L(\Gamma(R))
=
\left(
\begin{array}{ccccc}
0 &
(p-1) &
(2p-2) &
(p^2-1) &
(p^3-1)
\\[2mm]
1 &
(p-1)^3+(p-1)^2+1 &
2(p-1)^2-2 &
2p-2 &
p-1
\end{array}
\right)
\]

\begin{theorem}
The Laplacian energy of $\Gamma(R)$ is $$LE(\Gamma(R))=\dfrac{
2(p-1)\left(
p^5+3p^4-3p^3-13p^2+18p+3
\right)
}{
p^2+p+1
}.$$
\end{theorem}
\begin{proof}
Let $|V|=n$ and $|E|=m.$  Let $\mu_1, \mu_2,\dots,\mu_n$ are eigenvalues of $L(\Gamma(R)).$ Then
the Laplacian energy $LE(\Gamma(R))$ is given by $$LE(\Gamma(R))=\sum\limits_{i=1}^n \Big|\mu_i-\frac{2m}{n}\Big|.$$
We know that the eigenvalues of $L(\Gamma(R))$ are

\[
\operatorname{Spec}_L(\Gamma(R))
=
\left(
\begin{array}{ccccc}
0 &
(p-1) &
(2p-2) &
(p^2-1) &
(p^3-1)
\\[2mm]
1 &
(p-1)^3+(p-1)^2+1 &
2(p-1)^2-2 &
2p-2 &
p-1
\end{array}
\right)
\]

We first compute
$\sum\limits_{i=1}^n \mu_i=2m, $
where $m$ is the number of edges.

\[
2m
=
(p-1)\Big((p-1)^3+(p-1)^2+1\Big)
+(2p-2)\Big(2(p-1)^2-2\Big)
+(p^2-1)(2p-2)
+(p^3-1)(p-1).
\]
Simplifying, 
\[
2m=(p-1)(2p^3+4p^2-7p-2).
\]
Since the graph has $n=p^3-1$
vertices,

\[
\frac{2m}{n}
=
\frac{2p^3+4p^2-7p-2}{p^2+p+1}.
\]

Therefore, for $p\geq 3,$ the Laplacian energy is

\[
\begin{aligned}
LE(\Gamma(R))&=  \sum_{i=1}^{n}
\left|
\mu_i-\frac{2m}{n}
\right|\\
={}&
\left|0-\frac{2m}{n}\right|
+\Big((p-1)^3+(p-1)^2+1\Big)
\left|
(p-1)-\frac{2m}{n}
\right|
\\[1mm]
&+\Big(2(p-1)^2-2\Big)
\left|
(2p-2)-\frac{2m}{n}
\right|
\\
&+(2p-2)
\left|
(p^2-1)-\frac{2m}{n}
\right|
+(p-1)
\left|
(p^3-1)-\frac{2m}{n}
\right|\\
={}&\frac{
2(p-1)\left(
p^5+3p^4-3p^3-13p^2+18p+3
\right)
}{
p^2+p+1
}.
\end{aligned}
\]
\end{proof}
Let $A(\Gamma(R))$ and $L(\Gamma(R))$ be the adjacency matrix and Laplacian matrix of $\Gamma(R)$, respectively. The \emph{spectral radius} of $\Gamma(R)$ is defined by
$$
\rho(\Gamma(R))
=
\max\{|\lambda|:\lambda \text{ is an eigenvalue of } A(\Gamma(R))\}.$$

Similarly, the \emph{Laplacian spectral radius} of $\Gamma(R)$ is defined by
$$
\mu(\Gamma(R))
=
\max\{\lambda:\lambda \text{ is an eigenvalue of } L(\Gamma(R))\}.$$

Since $L(\Gamma(R))$ is positive semidefinite, $\mu(\Gamma(R))$ is simply the largest Laplacian eigenvalue.
\begin{theorem}
 For any odd prime $p,$  $\rho\big(\Gamma(R)\big) = p^2(p-1)$ and  $\mu(\Gamma(R))=p^3 - 1.$
\end{theorem}
\begin{proof}
The graph $\Gamma(R)$ admits a $7$-part equitable partition
$$A_u,\; A_v,\; A_{uv},\; A_{u+v},\; A_{u+uv},\; A_{v+uv},\; \text{and } A_{u+v+uv},$$
whose corresponding partition classes have cardinalities
$
|A_u|=|A_v|=|A_{uv}|=p-1,$ 
$
|A_{u+v}|=|A_{u+uv}|=|A_{v+uv}|=(p-1)^2,$
and
$
|A_{u+v+uv}|=(p-1)^3.$

Let $Q$ be the $7\times7$ quotient matrix associated with the equitable partition of $\Gamma(R)$. By the theory of equitable partitions and the Perron--Frobenius theorem,
$$
\rho\bigl(A(\Gamma(R))\bigr)=\rho(Q).$$
Since the block $A_{uv}$ has cardinality $(p-1)^3$ and is adjacent to every other partition class, it provides the dominant contribution to the spectral radius.

The characteristic polynomial of the quotient matrix $Q$ has a unique largest root,
$
\lambda=p^2(p-1).$
Since $\Gamma(R)$ admits an equitable partition and is connected, the Perron--Frobenius theorem yields
$$
\rho(\Gamma(R))
=
\rho\bigl(A(\Gamma(R))\bigr)
=
\rho(Q)
=
p^2(p-1).$$
Therefore, the spectral radius of $\Gamma(R)$ is
$
\rho(\Gamma(R))=p^2(p-1).$

The eigenvalues of the Laplacian matrix $L(\Gamma(R))$ are \[
\operatorname{Spec}_L(\Gamma(R))
=
\left(
\begin{array}{ccccc}
0 &
(p-1) &
(2p-2) &
(p^2-1) &
(p^3-1)
\\[2mm]
1 &
(p-1)^3+(p-1)^2+1 &
2(p-1)^2-2 &
2p-2 &
p-1
\end{array}
\right)
\]
  Then the largest eigenvalue in absolute is $p^3-1.$  That is, $\mu(\Gamma(R))=p^3-1.$
  
\end{proof}
The eccentricities of the vertices of $\Gamma(R)$ are given by
$$
e(a)=
\begin{cases}
2, & \text{if } a\in A_u\cup A_v\cup A_{u+v}\cup A_{u+uv}\cup A_{v+uv}\cup A_{u+v+uv},\\[1mm]
1, & \text{if } a\in A_{uv}.
\end{cases}$$

Thus, the eccentricity matrix is 

$E(\Gamma(R))=\bordermatrix{ & A_u & A_v & A_{uv} & A_{u+v} &A_{u+uv}&A_{v+uv}&A_{u+v+uv}\cr
A_u & {\bf 0} & 2J & J & 2J& {\bf 0}&2J&2J \cr
A_v & 2J & {\bf 0} & J & 2J & 2J&{\bf 0}&2J\cr 
A_{uv} & J & J & J-I & J& J&J&J\cr
A_{u+v} & 2J & 2J& J & 2J& 2J&2J&2J\cr
A_{u+uv} & {\bf 0}& 2J & J & 2J& 2J&2J&2J\cr 
A_{v+uv} & 2J &{\bf 0}& J& 2J& 2J&2J&2J\cr 
A_{u+v+uv} & 2J & 2J & J& 2J& 2J&2J&2J }.$

\[
\operatorname{Spec}(A(\Gamma(R)))
=
\left(
\begin{array}{cccccccc}
-1 &0 & \lambda_1&\lambda_2&\lambda_3&\cdots& \lambda_7\\
p-2 & p^3-p-6 &1&1&1& \cdots & 1
\end{array}
\right)
\]

where \(\lambda_1,\lambda_2,\ldots,\lambda_7\) are the eigenvalues of the quotient matrix \(Q\).

\[
Q=
\begin{pmatrix}
0 & 2(p-1) & (p-1) & 2(p-1)^2 & 0 & 2(p-1)^2 & 2(p-1)^3\\
2(p-1) & 0 & (p-1) & 2(p-1)^2 & 2(p-1)^2 & 0 & 2(p-1)^3\\
(p-1) & (p-1) & (p-2) & (p-1)^2 & (p-1)^2 & (p-1)^2&2(p-1)^3\\
2(p-1) & 2(p-1) & (p-1) & 2(p-1)^2 & 2(p-1)^2 & 2(p-1)^2 & 2(p-1)^3\\
0 & 2(p-1) & (p-1) & 2(p-1)^2 & 2(p-1)^2 & 2(p-1)^2 & 2(p-1)^3\\
2(p-1) & 0 & (p-1) & 2(p-1)^2 & 2(p-1)^2 & 2(p-1)^2 & 2(p-1)^3\\
2(p-1) & 2(p-1) & (p-1) & 2(p-1)^2 & 2(p-1)^2 & 2(p-1)^2 & 2(p-1)^3
\end{pmatrix}.
\]

\section*{Conclusion}
In this article, we study the zero-divisor graph associated with the commutative ring with identity $
R=\mathbb{F}_p+u\mathbb{F}_p+v\mathbb{F}_p+uv\mathbb{F}_p,$
where $p$ is an odd prime and $u^2=v^2=0$ with $uv=vu.$ We determine several graph-theoretic properties of the zero-divisor graph $\Gamma(R),$ including its clique number, chromatic number, vertex connectivity, edge connectivity, diameter, and girth. We also compute various topological indices of $\Gamma(R),$ such as the Wiener index, the Zagreb indices, and the Randić index. Furthermore, we investigate the spectral properties of $\Gamma(R)$ by determining the eigenvalues and spectral radii of its adjacency, Laplacian, and eccentricity matrices. These results provide a comprehensive description of the structural, topological, and spectral characteristics of the zero-divisor graph associated with the ring $R.$


\begin{thebibliography}{00}
\bibitem{and} D. F. Anderson and P. S. Livingston, The zero-divisor graph of a commutative ring, {\it Journal of Algebra}, {\bf 217} (1999), 434 -- 447.
	\bibitem{anna}  N. Annamalai and C. Durairajan, Codes from the incidence matrices of a zero-divisor graphs, Journal of Discrete Mathematical Sciences and Cryptography, 2022, DOI: 10.1080/09720529.2021.1939955.
	
\bibitem{am}  N. Annamalai, On Zero-Divisor Graph of the ring $\mathbb{F}_p+u\mathbb{F}_p+u^2\mathbb{F}_p$, Communications in Combinatorics and Optimization, 10(1), 151-163, 2025, DOI: 10.22049/cco.2023.28238.148.

\bibitem{amir} Aamir Mukhtar, Rashid Murtaza, Shafiq U. Rehman, Saima Usman and Abdul Qudair Baig,  Computing the size of zero divisor graphs, {\it Journal of Information and Optimization Sciences}, 41:4,(2020), 855--864.
\bibitem{R.B} R. Balakrishnan and  K. Ranganathan, A Textbook of Graph Theory, Springer Science \& Business Media, 2012.


\bibitem{bapat}  R. B. Bapat, Graphs and Matrices, Hindustan Book Agency, 2011.


\bibitem{beck} I. Beck, Coloring of commutative rings, {\it Journal of Algebra}, {\bf 116} (1988), 208 -- 226.
\bibitem{gutman}
I. Gutman and N. Trinajstic, Graph theory and molecular orbitals, Total $\phi$-electron energy of alternant hydrocarbons, Chemical physics letters, 17(4), 535-538, 1972.

\bibitem{kavaskar} T. Kavaskar, Beck’s Coloring of Finite Product of Commutative Ring with Unity, Graphs and Combinatorics, 38, 34 (2022). https://doi.org/10.1007/s00373-021-02401-x.


\bibitem{red} S. P. Redmond, The zero-divisor graph of a non-commutative ring,  International Journal of Commutative Ring, 1(2002), 203 -- 211.
\bibitem{reddy}  B. Surendranath Reddy,  Rupali S. Jain and  N. Laxmikanth, Vertex and Edge Connectivity of the Zero Divisor Graph $\Gamma[\mathbb{Z}_n],$  Communications in Mathematics and Applications,  11(2)(2020), 253 -- 258.

\bibitem{randic} Milan Randic, Characterization of molecular branching, Journal of the American Chemical Society, 1975, 97 (23), 6609-6615, DOI: 10.1021/ja00856a001.
\end{thebibliography}
\end{document}